\documentclass[11pt]{article}
\usepackage[utf8]{inputenc}
\usepackage[T1]{fontenc}
\usepackage{tgtermes}
\usepackage{amsfonts,amsmath,amssymb,amsthm}
\usepackage{enumerate}
\usepackage[linesnumbered,ruled,vlined]{algorithm2e}

\usepackage{capt-of}
\usepackage{color}
\usepackage{tikz}
\usepackage{url}

\usepackage[margin=1in]{geometry}

\newtheorem{theorem}{Theorem}[section]
\newtheorem{lemma}[theorem]{Lemma}
\newtheorem{claim}[theorem]{Claim}
\newtheorem{proposition}[theorem]{Proposition}

\newtheorem{problem}{Problem}

\newcommand{\NN}{\mathbb{N}}

\newcommand{\vozl}[1]{\operatorname{V}(#1)}
\newcommand{\neodv}[1]{\operatorname{\alpha}(#1)}
\newcommand{\klika}[1]{\operatorname{\omega}(#1)}

\newcommand{\oo}{\operatorname{o}}
\newcommand{\OO}{\operatorname{O}}
\newcommand{\OOm}{\operatorname{\Omega}}
\newcommand{\TT}{\operatorname{\Theta}}

\newcommand{\ie}{i.e.}

\title{Minimal normal graph covers}
\date{}
\author{David Gajser\thanks{A large part of the research was done when employed at IMFM, Department of Mathematics, Ljubljana, Slovenia and during a visit on Simon Fraser University which was supported by Javni sklad RS za razvoj kadrov in štipendije, 163. JR.}\\ FMF, Department of Mathematics\\ University of Ljubljana\\ Ljubljana, Slovenia 
\\{david.gajser{@}fmf.uni-lj.si}
\and
Bojan Mohar\thanks{Supported in part by an NSERC Discovery Grant (Canada),
   by the Canada Research Chair program, and by the
    Research Grant P1--0297 of ARRS (Slovenia).}~\thanks{On leave from:
    IMFM \& FMF, Department of Mathematics, University of Ljubljana, Ljubljana,
    Slovenia.}\\ Department of Mathematics\\ Simon Fraser University\\ Burnaby, BC~~V5A~1S6     
    \\{mohar{@}sfu.ca}
}

\begin{document}
\maketitle

\begin{abstract}
A graph is normal if it admits a clique cover $\mathcal C$ and a stable set cover $\mathcal S$ such that each clique in $\mathcal C$ and each stable set in $\mathcal S$ have a vertex in common. The pair $(\mathcal{C,S})$ is a normal cover of the graph.
We present the following extremal property of normal covers. For positive integers $c,s$, if a graph with $n$ vertices admits a normal cover with cliques of sizes at most $c$ and stable sets of sizes at most $s$, then $c+s\geq\log_2(n)$.
For infinitely many $n$, we also give a construction of a graph with $n$ vertices that admits a normal cover with cliques and stable sets of sizes less than $0.87\log_2(n)$.
Furthermore, we show that for all $n$, there exists a normal graph with $n$ vertices, clique number $\TT(\log_2(n))$ and independence number $\TT(\log_2(n))$.

When $c$ or $s$ are very small, we can describe all normal graphs with the largest possible number of vertices that allow a normal cover with cliques of sizes at most $c$ and stable sets of sizes at most $s$. However, such extremal graphs remain elusive even for moderately small values of $c$ and $s$.
\end{abstract}

\section{Introduction}

A graph $G$ is said to be \emph{normal} if it admits a set $\mathcal C$ of cliques and a set $\mathcal S$ of stable sets with the following properties:
\begin{itemize}
  \item[(1)] $\mathcal C$ is a cover of $G$, \ie, every vertex in $G$ belongs to one of the cliques in $\mathcal C$;
  \item[(2)] $\mathcal S$ is a cover of $G$, \ie, every vertex in $G$ belongs to one of the stable sets in $\mathcal S$;
  \item[(3)] Every clique in $\mathcal C$ and every stable set in $\mathcal S$ have a vertex in common.
\end{itemize}
We say that the pair $(\mathcal{C},\mathcal{S})$ is a \emph{normal cover} of $G$.

Clearly, a graph is normal if and only if its complement is normal.
This property is reminiscent on the notion of perfect graphs, see \cite{drugace,Korner73,Perfect_couples}. Namely, normality is one of the basic properties that every perfect graph satisfies. Of course, normality is much weaker condition since every odd cycle of length at least 9 is normal. Actually, De Simone and K\"orner conjectured \cite{zanimivaDomneva} that every graph without induced $C_5$, $C_7$, and the complement of $C_7$ is normal. Their conjecture was recently disproved \cite{Ararat}. Nevertheless, normal graphs make an interesting class of graphs.
Their importance lies in their close relationship to the notion of graph entropy, one of central concepts in information theory; see Csisz{\'a}r and K\"orner \cite{CsKor11} or \cite{CsKoLoMaSi90,Korner73,KorLongo73}.

For $c,s\in\NN$, we say that a normal cover is a \emph{$(c,s)$-normal cover}, if the sizes of cliques in the cover are at most $c$ and the sizes of stable sets in the cover are at most $s$. Furthermore, if a graph $G$ admits a $(c,s)$-normal cover, we say it is \emph{$(c,s)$-normal}. In this paper we consider the following extremal problem:
\begin{problem}
How large can a $(c,s)$-normal graph be?
\end{problem}
We define $N(c,s)$ as the maximum number of vertices of a $(c,s)$-normal graph. The obtained results that are reminiscent on the basics of Ramsey theory are far from conclusive.

Note first that it is not completely obvious that $N(c,s)$ is always finite. However, in Section~\ref{upper} we prove an upper bound given by the following theorem.

\begin{theorem}
 \label{upperBound}
If\/ $G$ is a $(c,s)$-normal graph, then $|G|\leq 2^{c+s}$, \ie, $N(c,s)\leq2^{c+s}$.
\end{theorem}

To state the theorem in different words, if $G$ has $n$ vertices and is $(c,s)$-normal, then $c+s\geq \log n$ where the logarithm has base 2, as it is assumed for all logarithms  in this paper. Hence, each normal cover of $G$ must have at least one \emph{element} (\ie, a stable set or a clique) of size $\OOm(\log n)$.

In analogy with perfect graphs, one might think that $N(c,s)$ should be closer to $c\cdot s$ (see \cite{Lovasz,CitatLovasza}). However, the bound in Theorem~\ref{upperBound} must be exponential, because in Section~\ref{lowerF} we present a construction that proves the following lower bounds.

\begin{theorem}
 \label{lowerBound}
For all integers $c\geq 6$, $N(c,c)\geq\sqrt{5}^{c}$. Additionally, $N(5,5)\geq 55$ and $N(4,4)\geq 25$.
\end{theorem}

To compare this bound with the bound in Theorem~\ref{upperBound}, note that $0.86 < \frac{1}{\log{\sqrt{5}}} < 0.87$. Hence, every normal cover of a graph on $n$ vertices must have a stable set or a clique of size at least $\tfrac{1}{2}\log n$, but for infinitely many $n$ we can construct a normal graph with a normal cover where all cliques and stable sets from the cover are of size less than $0.87\log n$.

To prove Theorem~\ref{lowerBound} we construct, for each positive integer $c$, a $(c,c)$-normal graph with roughly $\sqrt{5}^{c}$ vertices. We also present a different, simpler construction of a $(c,c)$-normal graph in Section~\ref{lowerG} which gives a bit weaker lower bound for $N(c,c)$. The construction  determines only some pairs of vertices to be connected by an edge or a non-edge. Any other pair of vertices can be adjacent or non-adjacent without affecting the $(c,c)$-normality.  If for those pairs we choose independently with probability $\tfrac{1}{2}$  whether they are connected by an edge or not, then we argue that for large $c$, we get small independence and clique numbers with positive probability. This way, we obtain in Section~\ref{lowerR} the following theorem.

\begin{theorem}
	\label{general}
For all $n$, there exists a normal graph $G$ with $n$ vertices that has clique number and independence number of order $\TT(\log n)$.
\end{theorem}

From Ramsey theory we know that every graph on $n$ vertices must contain a stable set or a clique of size $\frac{1}{2}\log n$.
Moreover, random graphs have independence number and clique number of order $\TT(\log n)$. Thus, Theorem~\ref{general} shows that normal graphs can have clique number and  independence number comparable to those of random graphs.
In relation to this, it should be mentioned that partial motivation for this study came from the question whether random graphs $G(n,p)$ are normal or not with high probability. As of today, this question is still unresolved.

Finally, in Section~\ref{exact} we describe, for small values of $c$ and $s$ ($c\le2$ or $s\le2$), all $(c,s)$-normal graphs with $N(c,s)$ vertices. For $c=3$ (or $s=3$) we derive a polynomial bound by showing that $N(3,s) = \Theta(s^3)$.

It is worth observing that the results of this paper are reminiscent to the basic facts about Ramsey numbers: the upper bound for $N(c,s)$ is exponential as is the lower bound for $N(c,c)$; the bases of exponents in the lower and upper bound are different and there is no clear evidence whether any of them is asymptotically tight. Small values of $N(c,s)$ are easy to obtain while the exact values for even a reasonably small values of $c$ and $s$ seem out of reach. Although our techniques are different from those in Ramsey theory, they exhibit the non-constructiveness familiar from Ramsey theory. While the lower bound is obtained by a direct, nonprobabilistic construction, the upper bound is proved by a non-combinatorial technique based on an algebraic argument using vector space dimension.

\section{Constructions of normal graphs with small cliques and stable sets}
	\label{lower}

In this section we prove Theorem~\ref{lowerBound} and Theorem~\ref{general}. For each integer $c>1$, we will first present a simple construction of a family of $(c,c)$-normal graphs with roughly $2^{c}$ vertices.
Then we will continue in two directions. For the first direction, we will show with a probabilistic argument that there is some graph in the constructed family of $(c,c)$-normal graphs that has independence number and clique number of size $\TT(c)$, which will prove Theorem~\ref{general}. For the second direction, we will enhance the construction to get a $(c,c)$-normal graph with at least $\sqrt{5}^{c}$ vertices, which will prove Theorem~\ref{lowerBound}.

\subsection{The construction of $G_c$}
	\label{lowerG}
A \emph{red-blue graph} is a graph  with each edge of color red or blue. When drawing such a graph, we will additionally make blue edges solid and red edges dashed. For an edge-colored graph $H$,  we say that a graph \emph{$G$ is of type $H$} if
\begin{itemize}
\item $\vozl{G}=\vozl{H}$,
\item each blue edge of $H$ is an edge of $G$,
\item no red edge of $H$ is an edge of $G$.
\end{itemize}
For a red-blue graph $H$ and for positive integers $c,s$, let $\mathcal{C}=\{C_1,C_2,\ldots, C_k\}$ be a cover of $H$ with blue cliques of size at most $c$ and let $\mathcal{S}=\{S_1,S_2,\ldots, S_l\}$ be a cover of $H$ with red cliques of size at most $s$. If the clique $C_i$ and the stable set $S_j$ intersect for any $i$ and $j$, then we say that the pair $(\mathcal{C},\mathcal{S})$ is a \emph{$(c,s)$-normal cover of $H$}. If $H$ admits a $(c,s)$-normal cover, we say that it is \emph{$(c,s)$-normal}. From the definition of graphs of type $H$ it is clear that if $(\mathcal{C},\mathcal{S})$ is a $(c,s)$-normal cover of $H$, then $(\mathcal{C},\mathcal{S})$ is also a $(c,s)$ normal cover of any graph of type $H$.

\begin{figure}[hbt]
	\begin{center}
	\begin{tikzpicture}[scale=1.8]
	\tikzstyle{vrata}=[circle,draw,inner sep=0pt, minimum size = 0.2em]
	\tikzstyle{modra}=[-, blue]
	\tikzstyle{rdeca}=[dashed,red,-]
	\node (s1) at ( -7,-1) [vrata, label = below:$v_1$] {};
	\node (s2) at ( -6,-1) [vrata, label = below:$v_2$] {};
	
	\node (p1) at ( -7,0) [vrata, label = above:$0$] {}
			edge [modra] (s1)
			edge [rdeca] (s2);
	\node (p2) at ( -6,0) [vrata, label = above:$1$] {}
			edge [modra] (s2)
			edge [rdeca] (s1);
	\end{tikzpicture}
	\end{center}
\caption{The red-blue graph $G_2$. Blue cliques $\{0,v_1\},\{1,v_2\}$ and red cliques $\{0,v_2\},\{1,v_1\}$ form a $(2,2)$-normal cover.}
\label{G2}
\end{figure}
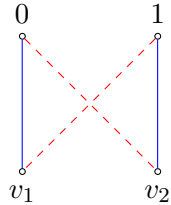

For an integer $c\geq 1$, define a red-blue graph $G_c$ recursively as $G_1=K_1$ and, for $c\geq 2$, $G_c$ is obtained from two copies of the graph $G_{c-1}$, denoted by $G_{c-1}(0)$ and $G_{c-1}(1)$, by adding two new vertices $0$ and $1$, such that each vertex of $G_{c-1}(0)$ is connected to $0$ by a blue edge and to $1$ by a red edge and each vertex of $G_{c-1}(1)$ is connected to $1$ by a blue edge and to $0$ by a red edge (see Figure~\ref{rekurzija}).

\begin{figure}[hbt]
\begin{center}
	\begin{tikzpicture}
	\tikzstyle{vrata}=[circle,draw,inner sep=0pt, minimum size = 0.2em]
	\tikzstyle{poddrevo}=[circle,draw,inner sep=0pt, shape border rotate=90, minimum height=5.5em,minimum width=5.5em]
	\tikzstyle{prazen}=[circle,inner sep=0pt, minimum size=3em]
		\node at ( 0,0) [poddrevo] (levi) {$G_{c-1}(0)$};
		\node at (3,0) [poddrevo] (desni)  {$G_{c-1}(1)$};
		\node (ne1) at ( 0,3) [vrata, label = left:$0$] {};
			\node (AA) at (0,0.3)[prazen]{};
			\draw[-, blue] (AA.north) edge (ne1);
			\draw[-, blue] (AA.north east) edge (ne1);
			\draw (-0.35,2)[blue] arc (250:290:1);
			\draw[-, blue] (AA.north west) edge (ne1);
			\node (AB) at (2.8,0.2)[prazen]{};
			\draw[dashed,red,-] (AB.north) edge (ne1);
			\draw[dashed,red,-] (AB.west) edge (ne1);
			\draw (0.6,2)[red,dashed] arc (295:335:1);
			\draw[dashed,red,-] (AB.north west) edge (ne1);
		\node (ne2) at ( 3,3) [vrata,label = left:$1$] {};
			\node (BB) at (3,0.3)[prazen]{};
			\draw[-, blue] (BB.north) edge (ne2);
			\draw[-, blue] (BB.north east) edge (ne2);
			\draw (2.65,2)[blue] arc (250:290:1);
			\draw[-, blue] (BB.north west) edge (ne2);
			\node (BA) at (0.2,0.2)[prazen]{};
			\draw[dashed,red,-] (BA.north) edge (ne2);
			\draw[dashed,red,-] (BA.north east) edge (ne2);
			\draw (2,2.5)[red,dashed] arc (205:245:1);
			\draw[dashed,red,-] (BA.east) edge (ne2);
	\end{tikzpicture}
	\caption{The recursive construction of $G_c$.}
	\label{rekurzija}
	\end{center}
\end{figure}
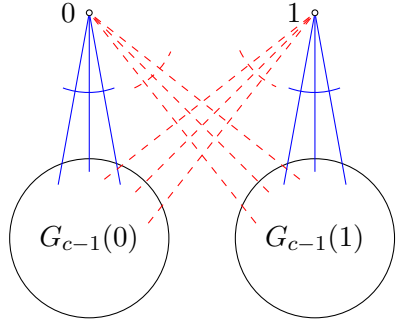

The graphs $G_2$, $G_3$ and $G_4$ are presented in Figures~\ref{G2}, \ref{G3} and \ref{G4}, respectively.

\begin{minipage}[hbt]{0.45\textwidth}
	\begin{center}
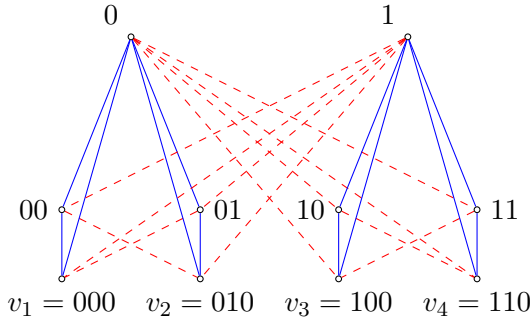

	\begin{tikzpicture}[scale=0.92]
	\tikzstyle{vrata}=[circle,draw,inner sep=0pt, minimum size = 0.2em]
	\tikzstyle{modra}=[-, blue]
	\tikzstyle{rdeca}=[dashed,red,-]
	\node (s1) at ( -7,-1) [vrata, label = below:${v_1=000}$] {};
	\node (s2) at ( -5,-1) [vrata, label = below:${v_2=010}$] {};
	\node (s3) at ( -3,-1) [vrata, label = below:${v_3=100}$] {};
	\node (s4) at ( -1,-1) [vrata, label = below:${v_4=110}$] {};
	\node (p1) at ( -7,0) [vrata, label =  left:$00$] {}
			edge [modra] (s1)
			edge [rdeca] (s2);
	\node (p2) at ( -5,0) [vrata, label =  right:$01$] {}
			edge [modra] (s2)
			edge [rdeca] (s1);
	\node (p3) at ( -3,0) [vrata, label =  left:$10$] {}
			edge [modra] (s3)
			edge [rdeca] (s4);
	\node (p4) at ( -1,0) [vrata, label =  right:$11$] {}
			edge [modra] (s4)
			edge [rdeca] (s3);
	\node (e1) at ( -6,2.5) [vrata, label = above left:$0$] {}
			edge [modra] (s1)
			edge [modra] (s2)
			edge [modra] (p1)
			edge [modra] (p2)
			edge [rdeca] (s3)
			edge [rdeca] (s4)
			edge [rdeca] (p3)
			edge [rdeca] (p4);
	\node (e2) at ( -2,2.5) [vrata, label = above left:$1$] {}
			edge [modra] (s3)
			edge [modra] (s4)
			edge [modra] (p3)
			edge [modra] (p4)
			edge [rdeca] (s1)
			edge [rdeca] (s2)
			edge [rdeca] (p1)
			edge [rdeca] (p2);
	\end{tikzpicture}
	\end{center}
\vspace{-20pt}
	\captionof{figure}{The graph $G_3$ with vertices labeled as binary sequences.}
	\label{G3}
\end{minipage} \hfill
\begin{minipage}[hbt]{0.45\textwidth}
	\begin{center}
	\begin{tikzpicture}[scale=0.45]
	\tikzstyle{vrata}=[circle,draw,inner sep=0pt, minimum size = 0.2em]
	\tikzstyle{modra}=[-, blue]
	\tikzstyle{rdeca}=[dashed,red,-]
	\node (s1) at ( -7,-1) [vrata, label = below:$v_1$] {};
	\node (s2) at ( -5,-1) [vrata, label = below:$v_2$] {};
	\node (s3) at ( -3,-1) [vrata, label = below:$v_3$] {};
	\node (s4) at ( -1,-1) [vrata, label = below:$v_4$] {};
	\node (s5) at ( 1,-1) [vrata, label = below:$v_5$] {};
	\node (s6) at ( 3,-1) [vrata, label = below:$v_6$] {};
	\node (s7) at ( 5,-1) [vrata, label = below:$v_7$] {};
	\node (s8) at ( 7,-1) [vrata, label = below:$v_8$] {};
	\node (p1) at ( -7,0) [vrata] {}
			edge [modra] (s1)
			edge [rdeca] (s2);
	\node (p2) at ( -5,0) [vrata] {}
			edge [modra] (s2)
			edge [rdeca] (s1);
	\node (p3) at ( -3,0) [vrata] {}
			edge [modra] (s3)
			edge [rdeca] (s4);
	\node (p4) at ( -1,0) [vrata] {}
			edge [modra] (s4)
			edge [rdeca] (s3);
	\node (p5) at ( 1,0) [vrata] {}
			edge [modra] (s5)
			edge [rdeca] (s6);
	\node (p6) at ( 3,0) [vrata] {}
			edge [modra] (s6)
			edge [rdeca] (s5);
	\node (p7) at ( 5,0) [vrata] {}
			edge [modra] (s7)
			edge [rdeca] (s8);
	\node (p8) at ( 7,0) [vrata] {}
			edge [modra] (s8)
			edge [rdeca] (s7);
	\node (e1) at ( -6,2.5) [vrata] {}
			edge [modra] (s1)
			edge [modra] (s2)
			edge [modra] (p1)
			edge [modra] (p2)
			edge [rdeca] (s3)
			edge [rdeca] (s4)
			edge [rdeca] (p3)
			edge [rdeca] (p4);
	\node (e2) at ( -2,2.5) [vrata] {}
			edge [modra] (s3)
			edge [modra] (s4)
			edge [modra] (p3)
			edge [modra] (p4)
			edge [rdeca] (s1)
			edge [rdeca] (s2)
			edge [rdeca] (p1)
			edge [rdeca] (p2);
	\node (e3) at (2,2.5) [vrata] {}
			edge [modra] (s5)
			edge [modra] (s6)
			edge [modra] (p5)
			edge [modra] (p6)
			edge [rdeca] (s7)
			edge [rdeca] (s8)
			edge [rdeca] (p7)
			edge [rdeca] (p8);
	\node (e4) at ( 6,2.5) [vrata] {}
			edge [modra] (s7)
			edge [modra] (s8)
			edge [modra] (p7)
			edge [modra] (p8)
			edge [rdeca] (s5)
			edge [rdeca] (s6)
			edge [rdeca] (p5)
			edge [rdeca] (p6);
	\node (m1) at ( -2,7) [vrata, label = above left:$0$] {}
			edge [modra] (s1)
			edge [modra] (s2)
			edge [modra] (s3)
			edge [modra] (s4)
			edge [modra] (p1)
			edge [modra] (p2)
			edge [modra] (p3)
			edge [modra] (p4)
			edge [modra] (e1)
			edge [modra] (e2)
			edge [rdeca] (s5)
			edge [rdeca] (s6)
			edge [rdeca] (s7)
			edge [rdeca] (s8)
			edge [rdeca] (p5)
			edge [rdeca] (p6)
			edge [rdeca] (p7)
			edge [rdeca] (p8)
			edge [rdeca] (e3)
			edge [rdeca] (e4);
	\node (m2) at ( 2,7) [vrata, label = above left:$1$] {}
			edge [modra] (s5)
			edge [modra] (s6)
			edge [modra] (s7)
			edge [modra] (s8)
			edge [modra] (p5)
			edge [modra] (p6)
			edge [modra] (p7)
			edge [modra] (p8)
			edge [modra] (e3)
			edge [modra] (e4)
			edge [rdeca] (s1)
			edge [rdeca] (s2)
			edge [rdeca] (s3)
			edge [rdeca] (s4)
			edge [rdeca] (p1)
			edge [rdeca] (p2)
			edge [rdeca] (p3)
			edge [rdeca] (p4)
			edge [rdeca] (e1)
			edge [rdeca] (e2);

	\fill[fill=white] (-4,0) ellipse (4.1 and 3);
	\fill[fill=white] (4,0) ellipse (4.1 and 3);
	\node (s1) at ( -7,-1) [vrata, label = below:$v_1$] {};
	\node (s2) at ( -5,-1) [vrata, label = below:$v_2$] {};
	\node (s3) at ( -3,-1) [vrata, label = below:$v_3$] {};
	\node (s4) at ( -1,-1) [vrata, label = below:$v_4$] {};
	\node (s5) at ( 1,-1) [vrata, label = below:$v_5$] {};
	\node (s6) at ( 3,-1) [vrata, label = below:$v_6$] {};
	\node (s7) at ( 5,-1) [vrata, label = below:$v_7$] {};
	\node (s8) at ( 7,-1) [vrata, label = below:$v_8$] {};
	\node (p1) at ( -7,0) [vrata] {}
			edge [modra] (s1)
			edge [rdeca] (s2);
	\node (p2) at ( -5,0) [vrata] {}
			edge [modra] (s2)
			edge [rdeca] (s1);
	\node (p3) at ( -3,0) [vrata] {}
			edge [modra] (s3)
			edge [rdeca] (s4);
	\node (p4) at ( -1,0) [vrata] {}
			edge [modra] (s4)
			edge [rdeca] (s3);
	\node (p5) at ( 1,0) [vrata] {}
			edge [modra] (s5)
			edge [rdeca] (s6);
	\node (p6) at ( 3,0) [vrata] {}
			edge [modra] (s6)
			edge [rdeca] (s5);
	\node (p7) at ( 5,0) [vrata] {}
			edge [modra] (s7)
			edge [rdeca] (s8);
	\node (p8) at ( 7,0) [vrata] {}
			edge [modra] (s8)
			edge [rdeca] (s7);
	\node (e1) at ( -6,2.5) [vrata] {}
			edge [modra] (s1)
			edge [modra] (s2)
			edge [modra] (p1)
			edge [modra] (p2)
			edge [rdeca] (s3)
			edge [rdeca] (s4)
			edge [rdeca] (p3)
			edge [rdeca] (p4);
	\node (e2) at ( -2,2.5) [vrata] {}
			edge [modra] (s3)
			edge [modra] (s4)
			edge [modra] (p3)
			edge [modra] (p4)
			edge [rdeca] (s1)
			edge [rdeca] (s2)
			edge [rdeca] (p1)
			edge [rdeca] (p2);
	\node (e3) at (2,2.5) [vrata] {}
			edge [modra] (s5)
			edge [modra] (s6)
			edge [modra] (p5)
			edge [modra] (p6)
			edge [rdeca] (s7)
			edge [rdeca] (s8)
			edge [rdeca] (p7)
			edge [rdeca] (p8);
	\node (e4) at ( 6,2.5) [vrata] {}
			edge [modra] (s7)
			edge [modra] (s8)
			edge [modra] (p7)
			edge [modra] (p8)
			edge [rdeca] (s5)
			edge [rdeca] (s6)
			edge [rdeca] (p5)
			edge [rdeca] (p6);
	\end{tikzpicture}
	\end{center}
\vspace{-30pt}
	
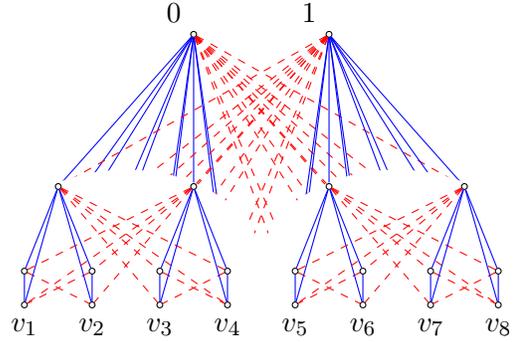
\captionof{figure}{The graph $G_4$. To make the picture more clear, the edges from $0$ and $1$ are not drawn completely.}
	\label{G4}
\end{minipage}
\vspace{5pt}

Next, we give an alternative, non-recursive description of $G_c$.
Let $\vozl{G_c}$ be the set of all binary sequences of length at least 1 and at most $c$ under the condition that all binary sequences of length exacly $c$ have to end with 0. Let $a=a_1 a_2 \ldots a_k\in\{0,1\}^k$ and $b=b_1 b_2 \ldots b_l\in\{0,1\}^l$ be two vertices of $G_c$. We may assume $k\leq l$. If $k=l$, then $a$ and $b$ are not adjacent. Otherwise, $ab$ is an edge if $b_1=a_1$, $b_2=a_2$ ,\ldots, $b_{k-1}=a_{k-1}$. The edge is blue if $a_k=b_k$ and red if $a_k\neq b_k$. The graph $G_3$ in Figure~\ref{G3} is labeled according to this description.
The proof that both descriptions of $G_c$ give the same graph is left to the reader.


\begin{proposition}
 \label{normalno}
For each positive integer $c$, the graph $G_c$ is $(c,c)$-normal and has $3\cdot2^{c-1}-2$ vertices.
\end{proposition}

\begin{proof}
From the description of $G_c$ by binary sequences, it is clear that
$$|\vozl{G_c}|=2\cdot2^{c-1}+2^{c-2}+2^{c-3}+\cdots + 2=3\cdot2^{c-1}-2.$$
Let $m=2^{c-1}$ and let $v_1,v_2,\ldots, v_m$ be the vertices that correspond to all binary sequences of length $c$ in $\vozl{G_c}$. We will use these vertices as ``generators'' for a $(c,c)$-normal cover of $G_c$. For each vertex $v_i=a_{i,1} a_{i,2}\ldots a_{i,c}$, define the set $C_i$ as the set of all beginnings of the sequence $v_i$, \ie, $C_i=\{a_{i,1} a_{i,2} \ldots a_{i,k}\,|\ 1\leq k\leq c\}$. It is clear from the definition of $G_c$  that $C_i$ is a blue clique of size $c$. Similarly, for each vertex $v_i$, define the set $S_i=\{a_{i,1} a_{i,2} \ldots  a_{i,k-1} \bar{a}_{i,k}\,|\ 1\leq k<c\}\cup\{v_i\}$, where we put a bar over $a_{i,k}$ to indicate $a_{i,k}\neq \bar{a}_{i,k}$. In other words, $S_i$ is the set of all sequences that are either $v_i$ or have the same beginning as $v_i$, but then differ from the beginning of $v_i$ by the last bit.  From the definition of $G_c$, we see that $S_i$ is a red clique of size $c$.

It is clear that each vertex of $G_c$ is in some blue clique $C_i$ and in some red clique $S_j$.
Moreover, each blue clique and each red clique intersect. To see this, take a blue clique $C_i$ and a red clique $S_j$. If $i=j$, then $C_i$ and $S_j$ bot contain $v_i$. If $i\neq j$, then let $k$ be the first index where the elements of sequences $v_i$ and $v_j$ are different. It follows that the sequence of the first $k$ elements of $v_i$ is in both $C_i$ and $S_j$.
Hence, the blue cliques $C_i$ and the red cliques $S_j$  form a $(c,c)$-normal cover of the graph $G_c$, which is thus $(c,c)$-normal.
\end{proof}

For the later use, the $(c,c)$-normal cover of the graph $G_c$ as described in the proof of Proposition~\ref{normalno} will be referred to as the \emph{standard normal cover of $G_c$}.

\subsection{Normal graphs with small clique number and independence number}
	\label{lowerR}

In the proof of Theorem~\ref{general} we will use the following two lemmas.

\begin{lemma}
	\label{mnozica}
For any large enough integer $c$, every subset $U$ of vertices of $G_c$ of size $4c$ has at least $4.1c^2$ pairs of vertices that are not connected by a blue edge and at least $4.1c^2$ pairs of vertices that are not connected by a red edge.
\end{lemma}

\begin{proof}
We will give the proof only for blue edges since the red-blue graph obtained from $G_c$ by interchanging red and blue edges is isomorphic to $G_c$ (this follows easily by inspection of the recursive construction).
In the proof, we will identify verices of $G_c$ with the corresponding binary sequences and we will
regularly use the term \emph{height of a vertex} to refer to the length of the sequence corresponding to this vertex.
We say that a vertex $v\in\vozl{G_c}$ is a \emph{predecessor} of a vertex $u\in\vozl{G_c}$, if $u$ and $v$ are connected by a blue edge and the height of $v$ is smaller than the height of $u$. We also say that $u$ is an ancestor of $v$. If $v$ is a  predecessor of $u$ and has height one less than $u$, we say that $v$ is a \emph{direct predecessor} of $u$ and $u$ is a \emph{direct ancestor} of $v$.

Let $U$ be a subset of vertices of $G_c$ that induces the most blue edges among all subsets of vertices of $G_c$ of size $4c$. We can assume the following three properties for the set $U$ (we may need to change $U$ to get these properties, but $U$ will remain a set with $4c$ vertices and it will induce the same number of blue edges as the current $U$).
\begin{enumerate}[a)]
\item \label{a}
{\it For each $u\in U$, all predecessors of $u$ are also in $U$}. If there is some predecessor $v$ of $u$ that is not in $U$, then because the neighbors (we only consider blue edges) of $u$ are also neighbors of $v$, we can replace $u$ by $v$ in $U$ without decreasing the number of blue edges in it. After we do such a transformation once, the sum of heights of vertices in $U$ strictly decreases so after finitely many such transformations all vertices of $U$ will have all their predecessors in $U$.

\item \label{b}
{\it If a vertex $v\in U$ of height $k$ has no direct ancestors in $U$, then all other vertices of $U$ of height at least $k$ have all their ancestors in $U$}. If this was not the case, then there exists a vertex $u\not\in U$ of height at least $k+1$ that is a direct ancestor of some vertex from $U$. By~\ref{a}), $u$ has strictly more neighbors in $U$ than $v$, thus if we substitute $v$ for $u$ in $U$ we get strictly more blue edges in $U$, which is a contradiction with the choice of $U$.

\item \label{c}
{\it For each $k$, $1\leq k < c-1$, $U$ has at most one vertex of height $k$ that does not have two direct ancestors}. If, for some $k$, there are two such vertices in $U$, say $u$ and $v$, then by~\ref{b}) it must be the case that $u$ and $v$ each has exactly one direct ancestor in $U$.
Hence, we can replace the direct ancestor of $v$ in $U$ and all of its ancestors in $U$ with the direct ancestor of $u$ which is not in $U$ and with its corresponding ancestors.
With this trasformation of $U$, the number of vertices in $U$ as well as the number of blue edges induced by $U$ is preserved. Furthermore, the assumption~\ref{a}) is preserved and, for each $l\neq k$, $1\leq l<c-1$, the number of vertices of height $l$ in $U$ that do not have two direct ancestors is also preserved. Because $v$ now has no direct ancestors in $U$, it must be the case that all other vertices of height $k$ in $U$ have exactly two direct ancestors in $U$ by~\ref{b}).
Because we can do the same transformation for all $k$, we proved that~\ref{c}) can indeed be assumed.
\end{enumerate}

For $i=1,2,\ldots,c$, let $x_i$ be the number of vertices of height $i$ in $U$. With assumptions~\ref{b}) and~\ref{c}) in hand, we can argue that $x_c$ is large. Let a vertex $v\in U$ be a lowest vertex from $U$ that has no direct ancestors and let $k$ be the height of $v$. If $k\geq c-1$, then using~\ref{b}) and~\ref{c}) we can prove by induction on $c-i$ that, for all $1\leq i<c$, it holds
$$x_i\leq \left\lceil\frac{x_c+1}{2^{c-i-1}}\right\rceil.$$
If $k<c-1$, then using  induction on $c-i$ we get
\begin{align*}
x_i&=\frac{x_c}{2^{c-i-1}} &&\textrm{if }k<i<c&&\textrm{ (use~\ref{a}) and~\ref{b}) to prove that)}
\end{align*}
and
\begin{align*}
x_i&\leq\left\lceil\frac{x_{k+1}/2+1}{2^{k-i}}\right\rceil &&\textrm{if } i\leq k &&\textrm{ (use~\ref{b}) and~\ref{c}) to prove that)}.
\end{align*}
\noindent In the second case, it furthermore holds
\begin{align}
	\label{xi_neenakost}
x_i\leq\left\lceil\frac{x_{k+1}/2+1}{2^{k-i}}\right\rceil \leq \left\lfloor\frac{x_{k+1}}{2^{k+1-i}}\right\rfloor+1.
\end{align}
The second inequality in~\eqref{xi_neenakost} is easy to see using the fact that $x_{k+1}$ is even, which follows from~\ref{b}).
In any case, we have
$$x_i\leq \frac{x_c}{2^{c-i-1}}+1$$
for all $i$, which gives
$$4c\leq x_c + \Big(x_c+1\Big)+\Big(\frac{x_c}{2}+1\Big)+\Big(\frac{x_c}{4}+1\Big)+\cdots+\Big(\frac{x_c}{2^{c-2}}+1\Big)< 3x_c+c$$
and hence $x_c > c$.

So we know that there are at least $c$ vertices in $U$ of height $c$, which implies that there are at least $c$ vertices in $U$ of height $c-1$ and $\lceil\frac{c}{2}\rceil$ vertices of height $c-2$ etc. Let $X\subseteq U$ be the set of all vertices in $U$ of height at least $c-4$. Because each vertex in $X$ has at most 22 possible ancestors, the set $X$ induces only $\OO(c)$ blue edges. Hence, there are at least $$\binom{2c+\lceil\frac{c}{2}\rceil+\lceil\frac{c}{4}\rceil+\lceil\frac{c}{8}\rceil}{2}-\OO(c)\geq4.13c^2-\OO(c)$$ pairs of vertices from $X$ not connected by a blue edge.
Hence, for large enough $c$, we have at least $4.1c^2$ pairs of vertices in $U$ that are not connected by a blue edge. This completes the proof.
\end{proof}


\begin{lemma}
	\label{generalLemma}
For every large enough integer $c$, there exists a graph $G$ of type $G_c$ such that $\neodv{G}<4c$ and $\klika{G}<4c$.
\end{lemma}

\begin{proof}
For a given integer $c$, let $G$ be a random graph of type $G_c$ which is obtained from $G_c$ by adding each missing edge of $G_c$ with probability $\frac{1}{2}$ independently for each edge, removing red edges and uncoloring blue edges. Let $n$ be the number of vertices of $G$. By Lemma~\ref{mnozica}, we get that, for large enough $c$,
\begin{align*}
\Pr\left[\klika{G}\geq 4c\right]
&\leq \sum_{\substack{U\subseteq \vozl{G}\\ |U|=4c}}\Pr\left[U\textrm{ induces a clique}\right]\\
&\leq \binom{n}{4c} 2^{-4.1c^2}.
\end{align*}
By Proposition~\ref{normalno} we have $c\geq\log{n}-1$. Therefore,
$$
\Pr\left[\klika{G}\geq 4c\right] \leq n^{4c}\, 2^{-4.1c\log n+4.1c} = 2^{4.1c}n^{-0.1c},
$$
which converges to 0 when $c$ tends to infinity. Hence, for large enough $c$,
$$\Pr\left[\klika{G}\geq 4c\right]<\frac{1}{2}.$$
From the symmetry between blue and red edges in the recursive definition of $G_c$ it follows that
$$\Pr\left[\neodv{G}\geq 4c\right]<\frac{1}{2}.$$
Hence, for large enough $c$,
$$\Pr\left[G \textrm{ has a stable set or a clique of size } 4c\right]<1.$$
This means that there exists a graph $G$ of type $G_c$ that has no clique or stable set of size $4c$.
\end{proof}

\begin{proof}[Proof of Theorem~\ref{general}.]
Let $n$ be a positive integer and let $c$ be such that $|\vozl{G_{c-1}}| < n \le |\vozl{G_c}|$. By Lemma~\ref{generalLemma}, there exists a graph $G$ of type $G_{c}$ such that $\neodv{G}=\OO(c)$ and $\klika{G}=\OO(c)$. We will make a serries of removals of vertices of $G$  until we get a graph with $n$, $n-1$ or $n-2$ vertices. Each time we will remove three vertices as can be seen in Figure~\ref{G42}. It is not hard to see that removing three by three vertices this way preserves normality of the graph.

\begin{figure}[hbt]
	\begin{center}
	\begin{tikzpicture}[scale=0.5]
	\tikzstyle{vrata}=[circle,draw,inner sep=0pt, minimum size = 0.2em]
	\tikzstyle{vrata2}=[circle,draw,inner sep=1pt, minimum size = 0.5em]
	\tikzstyle{modra}=[-, blue]
	\tikzstyle{rdeca}=[dashed,red,-]
	\node (s1) at ( -7,-1) [vrata, label = below:$v_1$, opacity=0.1] {};
	\node (s11) at ( -7,-1) [vrata2, opacity=0.5] {};
	\node (s2) at ( -5,-1) [vrata, label = below:$v_2$, opacity=0.1] {};
	\node (s22) at ( -5,-1) [vrata2, opacity=0.5] {};
	\node (s3) at ( -3,-1) [vrata, label = below:$v_3$, opacity=0.1] {};
	\node (s33) at ( -3,-1) [vrata2, opacity=0.5] {};
	\node (s4) at ( -1,-1) [vrata, label = below:$v_4$, opacity=0.1] {};
	\node (s44) at ( -1,-1) [vrata2, opacity=0.5] {};
	\node (s5) at ( 1,-1) [vrata, label = below:$v_5$, opacity=0.1] {};
	\node (s55) at ( 1,-1) [vrata2, opacity=0.5] {};
	\node (s6) at ( 3,-1) [vrata, label = below:$v_6$, opacity=0.1] {};
	\node (s66) at ( 3,-1) [vrata2, opacity=0.5] {};
	\node (s7) at ( 5,-1) [vrata, label = below:$v_7$] {};
	\node (s8) at ( 7,-1) [vrata, label = below:$v_8$] {};
	\node (p1) at ( -7,0) [vrata] {}
			edge [modra, opacity=0.1] (s1)
			edge [rdeca, opacity=0.1] (s2);
	\node (p2) at ( -5,0) [vrata, opacity=0.1] {}
			edge [modra, opacity=0.1] (s2)
			edge [rdeca, opacity=0.1] (s1);
	\node (p22) at ( -5,0) [vrata2, opacity=0.5] {};
	\node (p3) at ( -3,0) [vrata] {}
			edge [modra, opacity=0.1] (s3)
			edge [rdeca, opacity=0.1] (s4);
	\node (p4) at ( -1,0) [vrata, opacity=0.1] {}
			edge [modra, opacity=0.1] (s4)
			edge [rdeca, opacity=0.1] (s3);
	\node (p44) at ( -1,0) [vrata2, opacity=0.5] {};
	\node (p5) at ( 1,0) [vrata] {}
			edge [modra, opacity=0.1] (s5)
			edge [rdeca, opacity=0.1] (s6);
	\node (p6) at ( 3,0) [vrata, opacity=0.1] {}
			edge [modra, opacity=0.1] (s6)
			edge [rdeca, opacity=0.1] (s5);
	\node (p66) at ( 3,0) [vrata2, opacity=0.5] {};
	\node (p7) at ( 5,0) [vrata] {}
			edge [modra] (s7)
			edge [rdeca] (s8);
	\node (p8) at ( 7,0) [vrata] {}
			edge [modra] (s8)
			edge [rdeca] (s7);
	\node (e1) at ( -6,2.5) [vrata] {}
			edge [modra, opacity=0.1] (s1)
			edge [modra, opacity=0.1] (s2)
			edge [modra] (p1)
			edge [modra, opacity=0.1] (p2)
			edge [rdeca, opacity=0.1] (s3)
			edge [rdeca, opacity=0.1] (s4)
			edge [rdeca] (p3)
			edge [rdeca, opacity=0.1] (p4);
	\node (e2) at ( -2,2.5) [vrata] {}
			edge [modra, opacity=0.1] (s3)
			edge [modra, opacity=0.1] (s4)
			edge [modra] (p3)
			edge [modra, opacity=0.1] (p4)
			edge [rdeca, opacity=0.1] (s1)
			edge [rdeca, opacity=0.1] (s2)
			edge [rdeca] (p1)
			edge [rdeca, opacity=0.1] (p2);
	\node (e3) at (2,2.5) [vrata] {}
			edge [modra, opacity=0.1] (s5)
			edge [modra, opacity=0.1] (s6)
			edge [modra] (p5)
			edge [modra, opacity=0.1] (p6)
			edge [rdeca] (s7)
			edge [rdeca] (s8)
			edge [rdeca] (p7)
			edge [rdeca] (p8);
	\node (e4) at ( 6,2.5) [vrata] {}
			edge [modra] (s7)
			edge [modra] (s8)
			edge [modra] (p7)
			edge [modra] (p8)
			edge [rdeca, opacity=0.1] (s5)
			edge [rdeca, opacity=0.1] (s6)
			edge [rdeca] (p5)
			edge [rdeca, opacity=0.1] (p6);
	\node (m1) at ( -2,7) [vrata, label = above left:$0$] {}
			edge [modra, opacity=0.1] (s1)
			edge [modra, opacity=0.1] (s2)
			edge [modra, opacity=0.1] (s3)
			edge [modra, opacity=0.1] (s4)
			edge [modra] (p1)
			edge [modra, opacity=0.1] (p2)
			edge [modra] (p3)
			edge [modra, opacity=0.1] (p4)
			edge [modra] (e1)
			edge [modra] (e2)
			edge [rdeca, opacity=0.1] (s5)
			edge [rdeca, opacity=0.1] (s6)
			edge [rdeca] (s7)
			edge [rdeca] (s8)
			edge [rdeca] (p5)
			edge [rdeca, opacity=0.1] (p6)
			edge [rdeca] (p7)
			edge [rdeca] (p8)
			edge [rdeca] (e3)
			edge [rdeca] (e4);
	\node (m2) at ( 2,7) [vrata, label = above left:$1$] {}
			edge [modra, opacity=0.1] (s5)
			edge [modra, opacity=0.1] (s6)
			edge [modra] (s7)
			edge [modra] (s8)
			edge [modra] (p5)
			edge [modra, opacity=0.1] (p6)
			edge [modra] (p7)
			edge [modra] (p8)
			edge [modra] (e3)
			edge [modra] (e4)
			edge [rdeca, opacity=0.1] (s1)
			edge [rdeca, opacity=0.1] (s2)
			edge [rdeca, opacity=0.1] (s3)
			edge [rdeca, opacity=0.1] (s4)
			edge [rdeca] (p1)
			edge [rdeca, opacity=0.1] (p2)
			edge [rdeca] (p3)
			edge [rdeca, opacity=0.1] (p4)
			edge [rdeca] (e1)
			edge [rdeca] (e2);

	\fill[fill=white] (4,0) ellipse (4.1 and 3);
	\node (s5) at ( 1,-1) [vrata, label = below:$v_5$, opacity=0.1] {};
	\node (s55) at ( 1,-1) [vrata2, opacity=0.5] {};
	\node (s6) at ( 3,-1) [vrata, label = below:$v_6$, opacity=0.1] {};
	\node (s66) at ( 3,-1) [vrata2, opacity=0.5] {};
	\node (s7) at ( 5,-1) [vrata, label = below:$v_7$] {};
	\node (s8) at ( 7,-1) [vrata, label = below:$v_8$] {};
	\node (p5) at ( 1,0) [vrata] {}
			edge [modra, opacity=0.1] (s5)
			edge [rdeca, opacity=0.1] (s6);
	\node (p6) at ( 3,0) [vrata, opacity=0.1] {}
			edge [modra, opacity=0.1] (s6)
			edge [rdeca, opacity=0.1] (s5);
	\node (p66) at ( 3,0) [vrata2, opacity=0.5] {};
	\node (p7) at ( 5,0) [vrata] {}
			edge [modra] (s7)
			edge [rdeca] (s8);
	\node (p8) at ( 7,0) [vrata] {}
			edge [modra] (s8)
			edge [rdeca] (s7);
	\node (e3) at (2,2.5) [vrata] {}
			edge [modra, opacity=0.1] (s5)
			edge [modra, opacity=0.1] (s6)
			edge [modra] (p5)
			edge [modra, opacity=0.1] (p6)
			edge [rdeca] (s7)
			edge [rdeca] (s8)
			edge [rdeca] (p7)
			edge [rdeca] (p8);
	\node (e4) at ( 6,2.5) [vrata] {}
			edge [modra] (s7)
			edge [modra] (s8)
			edge [modra] (p7)
			edge [modra] (p8)
			edge [rdeca, opacity=0.1] (s5)
			edge [rdeca, opacity=0.1] (s6)
			edge [rdeca] (p5)
			edge [rdeca, opacity=0.1] (p6);
	\end{tikzpicture}
	\end{center}
\vspace{-30pt}
	\captionof{figure}{The graph $G_4$ with three triples of vertices (circled) removed. If we would also remove the fourth triple, we would get $G_3$.}
	\label{G42}
\end{figure}
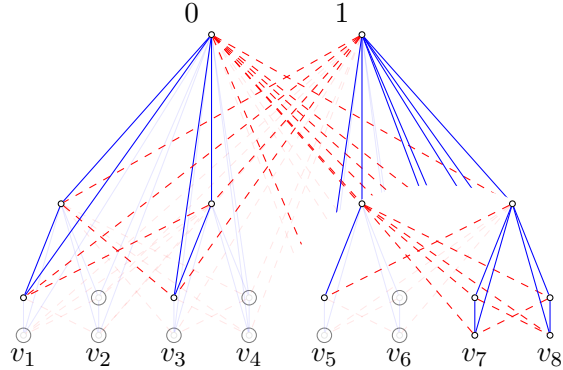

Let $G'$ be the resulting graph with $n$, $n-1$ or $n-2$ vertices. Since $G'$ is an induced subgraph of $G$, it holds $\neodv{G'}=\OO(c)$ and $\klika{G'}=\OO(c)$.
Next, we add at most two new vertices $C$ and $D$ to the graph $G'$, together with all possible edges from $C$ and $D$ except of the edge $CD$, to get an $n$-vertex graph $G''$. Clearly, this can increase $\alpha$ or $\omega$ by at most 2. We leave it to the reader to show that $G''$ is a normal graph. By Proposition~\ref{normalno} we know that $c=\TT(\log n)$. By our construction of the graph $G''$ it follows that $\neodv{G''}=\OOm(\log n)$ and $\klika{G''}=\OOm(\log n)$, hence $\neodv{G''}=\TT(\log n)$ and $\klika{G''}=\TT(\log n)$.
\end{proof}

\subsection{An improved lower bound for $N(c,c)$}
	\label{lowerF}
We have shown in Proposition~\ref{normalno} that $N(c,c)\geq 3\cdot2^{c-1}-2$. Here we improve this bound.

For a positive integer $c$, define a red-blue graph $F_c$ as $F_1=G_1$, $F_2=G_2$ and, for $c>2$, connect five copies of $F_{c-2}$ as shown in Figure~\ref{Fa}, where each edge between a copy of $F_{c-2}$ and a vertex in $\{A,B,C,D,E\}$ indicates the full red/blue join. Using induction on $c$, we can verify that the graphs $F_c$ are $(c,c)$-normal: a $(c,c)$-normal cover of $F_c$ consists of the union of normal covers of graphs $F_{c-2}$ where we add two vertices to each clique and two vertices to each stable set (which vertices we add can be seen from Figure~\ref{Fa}, there is only one option). We leave it to the reader to show that this is indeed a normal cover. We call it the ~\emph{standard normal cover of $F_c$}.

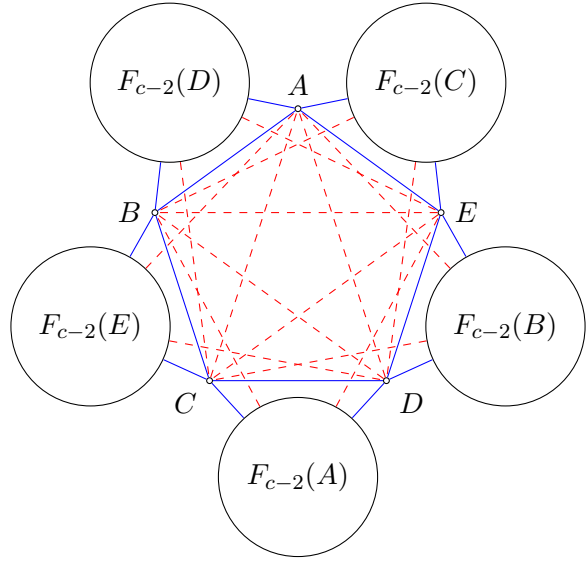
\begin{figure}[hbt]
\begin{center}
	\begin{tikzpicture}[scale=0.5]

\def \radius {4cm}
\def \radiuss {5.8cm}

	\tikzstyle{vrata}=[circle,draw,inner sep=0pt, minimum size = 0.2em]
	\tikzstyle{modra}=[-, blue]
	\tikzstyle{rdeca}=[dashed,red,-]
	\tikzstyle{poddrevo}=[circle,draw,inner sep=0pt, shape border rotate=90, minimum height=5.5em,minimum width=5.5em]

\node (A) [vrata, label = above:$A$] at ({360/5 * (1 - 1)+90}:\radius){};
\node (B) [vrata, label = left:$B$] at ({360/5 * (2 - 1)+90}:\radius){}
	edge [modra] (A);
\node (C) [vrata, label = below left:$C$] at ({360/5 * (3 - 1)+90}:\radius){}
	edge [rdeca] (A)
	edge [modra] (B);
\node (D) [vrata, label = below right:$D$] at ({360/5 * (4 - 1)+90}:\radius){}
	edge [modra] (C)
	edge [rdeca] (B)
	edge [rdeca] (A);
\node (E) [vrata, label = right:$E$] at ({360/5 * (5 - 1)+90}:\radius){}
	edge [modra] (D)
	edge [modra] (A)
	edge [rdeca] (C)
	edge [rdeca] (B);

\node[poddrevo] at ({360/5 * (1 - 1)+36+90}:\radiuss){$F_{c-2}(D)$}
	edge [modra] (A)
	edge [modra] (B)
	edge [rdeca] (C)
	edge [rdeca] (E);
\node[poddrevo] at ({360/5 * (2 - 1)+36+90}:\radiuss){$F_{c-2}(E)$}
	edge [modra] (B)
	edge [modra] (C)
	edge [rdeca] (A)
	edge [rdeca] (D);
\node[poddrevo] at ({360/5 * (3 - 1)+36+90}:\radiuss){$F_{c-2}(A)$}
	edge [modra] (C)
	edge [modra] (D)
	edge [rdeca] (B)
	edge [rdeca] (E);
\node[poddrevo] at ({360/5 * (4 - 1)+36+90}:\radiuss){$F_{c-2}(B)$}
	edge [modra] (D)
	edge [modra] (E)
	edge [rdeca] (C)
	edge [rdeca] (A);
\node[poddrevo] at ({360/5 * (5 - 1)+36+90}:\radiuss){$F_{c-2}(C)$}
	edge [modra] (E)
	edge [modra] (A)
	edge [rdeca] (D)
	edge [rdeca] (B);
\end{tikzpicture}
\caption{The recursive construction of $F_c$ from five copies of $F_{c-2}$. The blue edge between $A$ and $F_{c-2}(C)$ represents all possible blue edges between $A$ and vertices in $F_{c-2}(C)$. The same holds for any other edge shown in the picture.}
	\label{Fa}
	\end{center}
\end{figure}

Now we can prove Theorem~\ref{lowerBound}, which states:
\begin{center}
\textit{For all integers $c\geq 6$, $N(c,c)\geq\sqrt{5}^{c}$. Additionally, $N(5,5)\geq 55$ and $N(4,4)\geq 25$.}
\end{center}

\begin{proof}[Proof of Theorem~\ref{lowerBound}.]
If we denote $f_c=|\vozl{F_c}|$, then we have the recursion $f_1=1$, $f_2=4$, $f_c=5f_{c-2}+5$. This shows that $N(5,5)\geq 55$ and $N(4,4)\geq 25$. For  $c\geq 6$, the inequality $f_c\geq\sqrt{5}^{c}$ follows easily from $f_c=5f_{c-2}+5$ by induction (the basis is for $c=6$ and $c=7$).
\end{proof}

To conclude this section, let us discuss a special property that tells something about the optimality of the construction of graphs $F_c$ and prevents a simple operation that would improve the lower bound of Theorem \ref{lowerBound}.

\begin{proposition}
	\label{optimalno}
 For each positive integer $c$, if a subset of\/ $\vozl{F_c}$ of size at most $c$ does not induce a red edge, but intersects all red cliques of the standard normal cover of $F_c$, then it is of size $c$ and has at least $c-1$ common vertices with some blue clique of the standard normal cover of $F_c$.
\end{proposition}

Note that Proposition~\ref{optimalno} does not hold for graphs $G_c$. Take, for example, $G_5$ and let $X=\{0,1\}$ and $Y=\{00, 01, 10, 11\}$ be vertices of height 1 and 2, respectively. Note that $X$ and $Y$ both intersect all red and all blue cliques of the standard normal cover, while they do not induce any edge. We can connect vertices of $X$ with a blue edge and vertices of $Y$ with red edges and add a new vertex $v$ to $G_c$, connect it to $X$ with blue edges and to $Y$ with red edges. The new red-blue graph is also $(5,5)$-normal and has more vertices than $G_5$. It is clear that we could add in a similar way even two more vertices to increase the blue clique $\{0,1,v\}$ to size 5 and the graph would still be $(5,5)$-normal. 

\begin{proof}
For $c=1$ and $c=2$, the proposition is trivially true. Actually, as we will see below, the case $c=2$ is essentially the only reason for why it can happen that a set of size at most $c$ that does not induce a red edge, but intersects all red cliques of the standard normal cover of $F_c$, can have only $c-1$ common vertices with some blue clique of the standard normal cover of $F_c$.

Assume $c>2$ and assume that the proposition holds for $F_{c-2}$.
Let $U$ be a subset of $\vozl{F_c}$ of size at most $c$ that does not induce a red edge, but intersects all red cliques of the standard normal cover of $F_c$. Using notation from Figure~\ref{Fa}, $U$ contains at most two vertices from the set $\{A,B,C,D,E\}$, because any three vertices from this set induce at least one red edge.

If $U$ contains no vertex in $\{A,B,C,D,E\}$, then we have to cover all red cliques of each of five copies of $F_{c-2}$ independently. Hence, there exists a copy of $F_{c-2}$ for which we can cover all of its red cliques with a set of size at most $c/5$ that does not induce a red edge. Because the proposition holds for $F_{c-2}$, we have $c/5\geq c-2$, which means $c\leq 2$ (which is false). We conclude that $U$ contains at least one vertex from $\{A,B,C,D,E\}$.

If $U$ contains exactly one vertex from $\{A,B,C,D,E\}$, then we see in Figure~\ref{Fa} that there exist at least three copies of $F_{c-2}$ for which $U$ has to cover all their red cliques with at most $c-1$ vertices, meaning that there exists a copy of $F_{c-2}$ for which $U$ has to cover all of its red cliques with at most $(c-1)/3$ vertices (not inducing a red edge). But we assumed that the proposition holds for $F_{c-2}$, hence $(c-1)/3\geq c-2$, which means $c\leq 2$, which is false.

It follows that $U$ contains exactly 2 vertices from $\{A,B,C,D,E\}$.
Because of the symmetry we may assume that $U$ contains $A$ and $B$. By Figure~\ref{Fa}, $U\backslash\{A,B\}$ has to cover all stable sets of $F_{c-2}(D)$ and we know that it does not induce a red edge. By induction hypothesis, $U\backslash\{A,B\}$ is of size $c-2$ and has at least $c-3$ vertices in some blue clique $Q$ of the standard normal cover of $F_{c-2}(D)$. It follows that $U$ is of size $c$ and it has at least $c-1$ common vertices with the blue clique $Q\cup \{A,B\}$ that belongs to the standard normal cover of $F_c$.
\end{proof}

\section{An upper bound on $N(c,s)$}
\label{upper}

The red-blue graphs constructed in the previous section all admit $(c,c)$-normal covers for some $c=\TT(\log n)$.
In this section we prove that we cannot do better than that. More specifically, we prove Theorem~\ref{upperBound} confirming that $N(c,s)\leq2^{c+s}$.

Before going to the proof, we need some preparation. Let $G$ be a $(c,s)$-normal graph with a $(c,s)$-normal cover $(\mathcal{C},\mathcal{S})$. For a clique $C\in\mathcal{C}$, we say that a vertex $v\in C$ is \emph{private to} $C$ if it is in no other clique from $\mathcal{C}$. Similarly, for a stable set $S\in\mathcal{S}$, we say that a vertex $v\in S$ is \emph{private to} $S$ if it belongs to no other stable set from $\mathcal{S}$. We say that a $(c,s)$-normal cover $(\mathcal{C},\mathcal{S})$ of $G$ is \emph{minimal} if no proper subset of $\mathcal{C}$ is a clique cover and no proper subset of $\mathcal{S}$ is a stable set cover of $G$. Note that $(\mathcal{C},\mathcal{S})$ is minimal
if and only if each clique in $\mathcal{C}$ has a private vertex and each stable set in $\mathcal{S}$ has a private vertex.
Clearly, each $(c,s)$-normal graph has a minimal $(c,s)$-normal cover.

If we have a minimal normal cover of $G$, we can use the algorithm \textsc{C-reduce} to reduce the size of $G$ without affecting normality. Note that there is no condition in the algorithm on the input cover to be minimal.


\begin{algorithm}[h]
\DontPrintSemicolon
\KwIn{A graph $G$ with a normal cover $(\mathcal{C},\mathcal{S})$, a clique $C\in\mathcal{C}$ and a vertex $v$ that is private to $C$.}
	Let $P$ be the set of all private vertices of $C$.\;
	\If{$|P|=1$}{
		Remove $v$ from $G$ and from all stable sets in $\mathcal{S}$.\;
		Remove $C$ from $\mathcal{C}$.\;
	}
	\Else{Because $v\in P$ and $|P|\ne1$, there exists $v' \in P\setminus\{v\}$.\;
		 Remove all edges in $G$ that join $v$ or $v'$ with a vertex outside $C$.\;
		 Contract the edge $vv'$.\label{contract}
	}
\caption{\textsc{C-reduce}}
\label{alg1}
\end{algorithm}


\begin{claim}
	\label{algoritem}
The cover $(\mathcal{C},\mathcal{S})$ remains a normal cover of the resulting graph after the execution of \textsc{C-reduce}.
\end{claim}
\begin{proof}
If $|P|=1$, then it is clear that $\mathcal{C}$ remains a cover of $G-v$ by cliques. Because the clique $C$ was the only clique from $\mathcal{C}$ that included $v$, $(\mathcal{C},\mathcal{S})$ remains a normal cover.

If $|P|> 1$, then because $v$ and $v'$ were private to $C$, $\mathcal{C}$ remains a cover of $G$ by cliques. Furthermore,  all stable sets in $\mathcal{S}$ remain stable sets, since none of them contains both $v$ and $\tilde v$, which are adjacent in $G$, and all edges from $\{v,v'\}$ to vertices out of $C$ have been removed.
\end{proof}

Let $G$ be a $(c,s)$-normal graph with a minimal $(c,s)$-normal cover $(\mathcal{C},\mathcal{S})$. We say that vertices \linebreak$v_1,v_2,\ldots, v_{|\mathcal{C}|}\in\vozl{G}$ are \emph{clique generators} for $\mathcal{C}$ if there is a bijective correspondence between vertices $v_i$ and cliques $C_i\in\mathcal{C}$ such that for each $i$, $v_i$ is private to $C_i$. We define \emph{stable set generators} for $\mathcal{S}$ analogously.

In what follows, we prove Theorem~\ref{upperBound} by induction on $c+s$. The strategy is as follows. For a $(c,s)$-normal graph $G$, we first take a minimal $(c,s)$-normal cover $(\mathcal{C},\mathcal{S})$. Then we find a set of clique generators and we run \textsc{C-reduce} for all these generators, \ie, we repeat the reduction $|\mathcal{C}|$ times\footnote{In the formal proof of Theorem~\ref{upperBound} below, we reduce clique generators only if we have at least as many stable sets as cliques. Else, we reduce stable set generators.}. Then we are left with a $(c-1,s)$-normal graph $G'$. If $G'$ has at least half of the original vertices, we can use induction hypothesis to show that $N(c,s)\leq2N(c-1,s)\leq2^{c+s}$.

In order for $G$ not to lose too many vertices during the just described reduction of clique generators, we want to prove that in any minimal normal cover we do not have too many stable sets and cliques. To prove this, we will use the following lemma.

\begin{lemma}
\label{matt}
Let $(\mathcal{C},\mathcal{S})$ be a minimal $(c,s)$-normal cover of a $(c,s)$-normal graph $G$ on $n>1$ vertices. If $U\subseteq\vozl{G}$ is a set of clique generators that is also a set of stable set generators of $(\mathcal{C},\mathcal{S})$, then $n\geq 2|U|$.
\end{lemma}

\begin{proof}
From the definition of generators it follows that  $|\mathcal{C}|=|\mathcal{S}|=|U|$. If we denote $t=|U|$, let $U=\{v_1,v_2,\ldots, v_t\}$ and let the remaining vertices of $G$ be $v_{t+1},v_{t+2},\ldots, v_n$. Furthermore, let the (unique) clique and the (unique) stable set from $(\mathcal{C},\mathcal{S})$ that contains $v_i$ be denoted by $C_i$ and $S_i$, respectively.
For $i=1,\ldots, n$, we define column vectors $c_i,s_i \in\{0,1\}^t$ as
$$
c_i(j)=\left\{
	\begin{array}{ll}
		1 & v_i\in C_j\\
		0 & \mbox{otherwise}
	\end{array}
\right. \qquad \mbox{and} \qquad
s_i(j)=\left\{
	\begin{array}{ll}
		1 & v_i\in S_j\\
		0 & \mbox{otherwise.}
	\end{array}
\right.
$$
This gives us, for each $i$,
$$c_i\cdot s^T_i[j,k]=\left\{
	\begin{array}{ll}
		1  & \mbox{if } C_j \cap S_k = \{v_i\}\\
		0 & \mbox{otherwise.}
	\end{array}
\right.$$
Because $(\mathcal{C},\mathcal{S})$ is a normal cover, we have
$$\sum_{i=1}^n c_i\cdot s^T_i=J_t,$$
where $J_t$ denotes the all-ones $t\times t$ matrix. Furthermore,
$$\sum_{i=1}^t c_i\cdot s^T_i=I_t,$$
where $I_t$ denotes the identity matrix. It follows that
$$\sum_{i=t+1}^n c_i\cdot s^T_i=J_t-I_t.$$
On the left we have a sum of $n-t$ matrices of rank 1 and on the right we have a matrix of rank $t$. Here we used the assumption that $n>1$, which implies that $t>1$. It follows that $n-t\geq t$ which proves the lemma.
\end{proof}

Now we can prove that no minimal normal cover contains too many stable sets or cliques.

\begin{lemma}
\label{central}
Let $(\mathcal{C},\mathcal{S})$ be a minimal $(c,s)$-normal cover of a graph $G$ of order $n$.
Then $|\mathcal{C}|+|\mathcal{S}|\leq n+1$.
\end{lemma}

\begin{proof}
We may assume without loss of generality that $|\mathcal{C}|\geq|\mathcal{S}|$.
We will prove the lemma  by induction on $n$, starting with $n=1$ where the lemma clearly holds. For $n>1$, exactly one of the following two options has to hold.

Suppose first that each clique $C_i\in\mathcal{C}$ contains a vertex $v_i$ which is private to $C_i$ and is also private to a stable set, which we denote by $S_i\in\mathcal{S}$, such that $v_i$ is the unique private vertex of $S_i$.  If we denote by $U$ the set of vertices $v_i$, then $U$ is a set of clique generators for $(\mathcal{C},\mathcal{S})$. Because $|\mathcal{C}|\geq|\mathcal{S}|$, the set $U$ is also a set of stable set generators for $(\mathcal{C},\mathcal{S})$ and, in particular, $|\mathcal{S}| = |U|$.
This implies that $2|U|\leq n$ by Lemma~\ref{matt} and hence $|\mathcal{C}|+|\mathcal{S}|\leq n$.

Suppose now that there exists a clique $C\in\mathcal{C}$ such that all of its private vertices are covered with stable sets, each of which has some private vertex outside $C$.
If we run \textsc{C-reduce} on the clique $C$ and one of its private vertices, we get a normal graph $G$ with one vertex less. Because the stable sets that contain private vertices of $C$ had at least one private vertex outside $C$, the number of stable sets in the new minimal normal cover is not reduced while the number of cliques is reduced by at most one. The result follows by induction.
\end{proof}

We are ready to prove Theorem~\ref{upperBound}. Actually, we will prove a slightly stronger version in order to use induction.

\begin{lemma}
$N(c,s) \le 2^{c+s}-1$.
\end{lemma}

\begin{proof}
The proof is by induction on $c+s$.
Let $G$ be a $(c,s)$-normal graph of order $n$.
If $c=1$ or $s=1$, then $G$ is an empty graph or a complete graph, hence $n=c+s-1\leq 2^{c+s}-1$. So suppose $c,s\geq 2$ and let $(\mathcal{C},\mathcal{S})$ be a minimal $(c,s)$-normal cover of $G$. By Lemma~\ref{central} we have $|\mathcal{C}|+|\mathcal{S}|\leq n+1$ and hence $|\mathcal{C}|\leq\frac{n+1}{2}$ or $|\mathcal{S}|\leq\frac{n+1}{2}$. We may assume that $|\mathcal{C}|\leq|\mathcal{S}|$ and therefore $|\mathcal{C}|\leq\frac{n+1}{2}$ because if this is not the case, then we can replace $G$ by the complement of $G$. Next, we run \textsc{C-reduce} $|\mathcal{C}|$-times, once for each clique in $\mathcal{C}$ and for some private vertex of this clique. By Claim~\ref{algoritem} we are left with a $(c-1,s)$-normal graph on at least $\frac{n-1}{2}$ vertices which implies that $\frac{n-1}{2}\leq N(c-1,s)$. By the induction hypothesis, we have $N(c-1,s)\leq 2^{c+s-1}-1$ which implies $n\leq 2^{c+s}-1$.
\end{proof}

\section{Some exact values of $N(c,s)$}
	\label{exact}

We start by a simple observation.
\begin{claim}
 \label{maksim}
If $(\mathcal{C},\mathcal{S})$ is a minimal $(c,s)$-normal cover of a graph $G$ with $N(c,s)$ vertices, then each clique in $\mathcal{C}$ has $c$ vertices and each stable set in $\mathcal{S}$ has $s$ vertices.
\end{claim}

\begin{proof}
Because of the symmetry between cliques and stable sets, it suffices to prove the claim for cliques. Suppose that there is a clique $C\in \mathcal{C}$ of size strictly less than $c$ with a private vertex $v$. If we add a new vertex $u$ to $G$ and connect it to all vertices of $C$, we can increase the clique $C$ by adding the vertex $u$. Furthermore, if we clone each stable set from $\mathcal{S}$ that contains $v$ and replace $v$ with $u$ in one of the clones, the resulting cover is clearly a $(c,s)$-normal cover of the (new) graph $G$. Hence, we have a $(c,s)$-normal graph with more than $N(c,s)$ vertices, a contradiction.
\end{proof}

For integers $r\ge2$ and $k\ge1$, let a red-blue graph $G_{r,k}$ be the following graph on $n=r(k+1)$ vertices. The blue subgraph of $G_{r,k}$ is the disjoint union of $r$ stars, each isomorphic to $K_{1,k}$. The vertices of degree $k$ in these stars are called the \emph{roots}. For each star, add red edges to form a red clique on the non-root vertices of the star, and for each root vertex $u$, add red edges to every non-root vertex in all other stars. If $r\ge3$, we also add red edges between each pair of roots.

It is easy to see that the graph $G_{r,k}$ is $(2,k+r-1)$-normal: take each blue edge for a clique of the cover and for each root, take all of its blue neighbors together with all of the remaining roots for a stable set. If we denote by $m=k+r-1$ the size of stable sets in our cover, we have $n=(k+1)(m-k+1) = \left(\frac{m}{2}+1\right)^2 - (k-\frac{m}{2})^2$. Hence, for $k= \left\lfloor\frac{m}{2}\right\rfloor$ and $k= \left\lceil\frac{m}{2}\right\rceil$, we have $n=\bigl\lfloor\left(\frac{m}{2}+1\right)^2\bigr\rfloor$.

Next, we will compute $N(c,s)$ and we will describe $(c,s)$-normal graphs with $N(c,s)$ vertices for small values of $c$ and $s$. Note that $N(c,s)=N(s,c)$ and for every $(c,s)$-normal graph with $N(c,s)$ vertices, its complement is an $(s,c)$-normal graph with $N(s,c)$ vertices. 

Let $G$ be a $(c,s)$-normal graph with $N(c,s)$ vertices and minimum number of edges. Let $(\mathcal{C,S})$ be a minimal $(c,s)$-normal cover of $G$. Together with
Claim~\ref{maksim} we therefore have the following properties of $G$:

\begin{enumerate}
  \item Every clique in $\mathcal{C}$ has $c$ vertices.
  \item Every stable set in $\mathcal S$ has $s$ vertices.
  \item Every edge of $G$ is in one of the cliques in $\mathcal C$.
  \item Every clique in $\mathcal C$ has a private vertex. The degree of every vertex that is private to a clique in $\mathcal C$ is $c-1$.
  \item Every stable set in $\mathcal S$ has a private vertex.
\end{enumerate}

\begin{proposition}
\label{prop:small values}
Let $m\ge1$ be an integer.
\begin{enumerate}[a)]
\item $N(1,m)=N(m,1)=m$ and the only $(m,1)$-normal graph with $N(m,1)$ vertices is the complete graph $K_m$.
\item\label{bbb} $N(2,m)=N(m,2)=\bigl\lfloor\left(\frac{m}{2}+1\right)^2\bigr\rfloor$ and $(2,m)$-normal graphs with $N(2,m)$ vertices are precisely graphs of type $G_{r,k}$\/, where $r=m+1-k$ and $k = \left\lfloor\frac{m}{2}\right\rfloor$ or $k = \left\lceil\frac{m}{2}\right\rceil$.
\item \label{tockaC} $N(3,3)=10$ and every $(3,3)$-normal graph with $N(3,3)$ vertices is of type $G_3$ or $F_3$.
\item \label{tockaD} For every $s\ge4$ we have
$$
  \tfrac{1}{27}s^3 (1+\oo(1)) \le N(3,s) \le \tfrac{4}{27}s^3 (1+\oo(1)).
$$
\end{enumerate}
\end{proposition}

\begin{proof}
\begin{enumerate}[a)]
\item Clearly, in a graph that has a normal cover with stable sets of size 1, every clique must contain all vertices, so the graph is complete. Thus, $N(1,m)=N(m,1)=m$.

\item Let $(\mathcal{C},\mathcal{S})$ be a minimal $(2,m)$-normal cover of a graph $G$ with $n=N(2,m)$ vertices, let $Q$ be a set of clique generators and let $R=\vozl{G}\backslash Q$ be the set of the remaining vertices.
Let $v$ be a vertex in $R$ that is covered by the largest number of cliques from $\mathcal{C}$, say by $k$ cliques. Since each of these $k$ cliques is of size 2, any stable set in the cover that does not contain $v$ must cover $k$ private vertices of cliques that cover $v$ together with at least $|R|-1$ additional vertices to intersect each clique from $\mathcal{C}$, hence $m\geq k+|R|-1$. Because no vertex is covered with more than $k$ cliques, we have $n\leq|R|+k|R|$ which gives
$$
   n\leq(k+1)(m+1-k) = \left(\frac{m}{2}+1\right)^2 - \left(k-\frac{m}{2}\right)^2.
$$
The right-hand side is maximal only for $k = \left\lfloor\frac{m}{2}\right\rfloor$ or $k = \left\lceil\frac{m}{2}\right\rceil$ in which case we get
$n\leq\bigl\lfloor\left(\frac{m}{2}+1\right)^2\bigr\rfloor$.
The graph $G_{r,k}$ with parameters as in the statement of the proposition shows that $N(2,m) \geq \allowbreak \bigl\lfloor\left(\frac{m}{2}+1\right)^2\bigr\rfloor$ which proves $N(2,m)=\bigl\lfloor\left(\frac{m}{2}+1\right)^2\bigr\rfloor$. It is now clear that every $(2,m)$-normal graph with $N(2,m)$ vertices is of type $G_{r,k}$ (and vice versa).

\item The graphs of type $G_3$ (see Figure~\ref{G3}) and the graphs of type $F_3$ (see Figure~\ref{Fa}) are examples of $(3,3)$-normal graphs with 10 vertices. In what follows, we will prove that that a $(3,3)$-normal graph cannot have more than 10 vertices and if it has 10 vertices, then it is of type $G_3$ or $F_3$. Let us first introduce an auxiliary graph $H$ and some notation that will also be used in the proof of~\ref{tockaD}).

For $s\ge3$, let $G$ be a $(3,s)$-normal graph with $N(3,s)$ vertices, and let $(\mathcal{C},\mathcal{S})$ be a minimal $(3,s)$-normal cover of $G$. We remove all edges from $G$ except those induced by cliques from $\mathcal{C}$. Note that doing so we did not lose the information about the ``normal structure of $G$'', which is hidden in the normal cover $(\mathcal{C},\mathcal{S})$.
For each $C\in\mathcal C$ we select one vertex that is private to $C$ and denote it by $p_C$. Let $P=\{p_C\mid c\in\mathcal C \}$. By the assumption outlined above, vertices in $P$ have degree 2 in $G$ and it is easy to see that they form an independent set in $G$. Let $H=G-P$. Note that $H$ has no isolated vertices. Each edge $e\in E(H)$ corresponds to at least one clique in $\mathcal C$. We denote by $m(e)\ge1$ the number of cliques in $\mathcal C$ containing $e$, and call it the \emph{multiplicity} of $e$ in $H$. Clearly, $H$ and its multiplicity function determine $G$ up to an isomorphism. Moreover,
\begin{equation}
   |G| = N(3,s) = |H| + \sum_{e\in E(H)} m(e).
   \label{eq:number vertices in G from H}
\end{equation}

For each vertex $v$ of $H$ define $\lambda(v)=\max \{m(e)\mid e\in E(H) \text{ is incident with } v\}$. For a vertex $v\in V(G)$, let $u\in V(G)$ be such that $\lambda(v)=m(vu)$. Consider a stable set $S\in\mathcal S$ containing the private vertex $w$ in a clique $\{u,v,w\}$. Since $S$ intersects all cliques in $\mathcal C$ and is independent, it must contain $m(uv)$ private vertices of the cliques containing $u$ and $v$, and must contain a different vertex for each neighbor $u'\ne u$ of $v$ in $H$ (either contains $u'$ or the corresponding private vertices). This shows that
\begin{equation}
   \deg_H(v) + \lambda(v) -1 \le s.
   \label{eq:degree in H is bounded}
\end{equation}

Let us now consider the case $s=3$.  Note that $n=|G|\geq 10$ and all cliques in $\mathcal{C}$ and all stable sets in $\mathcal{S}$ are of size 3. 
By~\eqref{eq:degree in H is bounded}, the multiplicities in $H$ are at most 3. If an edge $uv$ has multiplicity 3, then considering a stable set that covers the private vertex of a clique from $\mathcal{C}$ that contains $u$ and $v$, it follows that $G$ has just 5 vertices. So, this cannot happen.

Suppose now that $m(uv)=2$. Then there is a vertex $z$ of $H$ that is incident with all edges of $H-uv$. If $H$ had just 3 vertices, \eqref{eq:number vertices in G from H} would imply that $G$ has at most 9 vertices. Thus, there is a vertex $z'\notin \{u,v,z\}$. Since $H$ has no isolated vertices and $z$ covers all edges in $H-uv$, the vertex $z'$ is adjacent to $z$ and has degree 1 in $H$. The edges $vz$ and $uz$ (if they are present in $H$) have multiplicity 1 since no vertex covers all edges of $H-uz$ or $H-vz$. Using~\eqref{eq:number vertices in G from H} and $|G|\geq 10$, we see that $H$ must have another vertex or both edges $uz$ and $vz$. Hence, the edge $zz'$ must also have multiplicity 1, because neither $u$ nor $v$ (nor any other vertex) covers all edges of $H-zz'$. This implies that $z$ has another neighbor $z''$, which is also of degree 1 in $H$. Hence, $z$ is not adjacent to $u,v$ because there cannot exist a stable set of size 3 that would cover the private vertex of the edge $uz$ or $vz$ and would intersect all cliques from $\mathcal{C}$. Considering~\eqref{eq:number vertices in G from H} and $|G|\geq 10$, we see that $z$ has yet another neighbor $z'''$. However, this is impossible since any stable set that would cover the private vertex corresponding to the clique containing $zz'$ and would intersect all cliques from $\mathcal{C}$, would be of size at least 4. The conclusion is that all edges in $H$ have multiplicity 1.

By~\eqref{eq:degree in H is bounded}, the degree of every vertex in $H$ is at most 3. If a vertex $v$ has degree 3, then $H$ has only one connected component and no other vertex than $v$ and its neighbors. If there was another vertex in $H$ adjacent to a neighbor $u$ of $v$, then a stable set that would cover the private vertex of the clique corresponding to $uv$ could not cover all cliques from $\mathcal{C}$. Thus by~\eqref{eq:number vertices in G from H}, $H$ needs to have at least 6 edges in order for $G$ to have at least 10 vertices. This implies that $H=K_4$ which is clearly impossible. Hence, all vertices of $H$ have degree 2 or 1. It follows that $H$ is a disjoint union of cycles and paths.

If $H$ has a connected component $K_2$, then $G$ has at most $3+N(3,2)=9$ vertices, which is not the case. It is clear that $H$ has at most 3 connected components and because 3 connected components imply that $G$ is a disjoint union of three triangles and has at most 9 vertices, $H$ has at most 2 connected components. If $H$ has two connected components, we know that none is $K_2$. Hence, a stable set from $\mathcal{S}$ needs to cover all edges of one component with just one vertex. This implies that both components are paths of length 2, which gives that $G$ is of type $G_3$.

If $H$ has only one component, we know that it is a path or a cycle. If $H$ is a path of length at most 4, than $G$ has at most 9 vertices, which is false. If $H$ is a path of length at least 5, then the stable set that covers the private vertex corresponding to the second edge in the path cannot cover all cliques from $\mathcal{C}$. Hence, $H$ is a cycle. It is easy to see that it has to be of length 5, which gives that $G$ is of type $F_3$.

\item To prove the lower bound, let $d=\left\lceil \frac{s}{3} \right\rceil$. Consider the graph $H$ consisting of $s-2d+2$ copies of the star $K_{1,d}$, each edge of which has multiplicity $d$. The corresponding graph $G$ -- see the proof in~\ref{tockaC}) -- is $(3,s)$-normal. By \eqref{eq:number vertices in G from H}, it has $(s-2d+2)(d^2+d+1)$ vertices, which gives the desired lower bound.

As for the upper bound, define $H$ and its multiplicity function $m$ as described above. Let $\lambda$ be the largest multiplicity over all edges of $H$; suppose it is attained for the edge $uv$. Consider a stable set containing private vertices of the cliques corresponding to $uv$. As discussed for the bound~\eqref{eq:degree in H is bounded}, it contains $s-\lambda$ other vertices $u_1,\dots,u_{s-\lambda}$. Suppose that exactly first $k$ of them are private vertices of cliques in $\mathcal{C}$ that have been removed when defining $H$. The vertices $u_{k+1},\dots,u_{s-\lambda}$ cover all edges of $H$ except $uv$ and possibly the edges that correspond to private vertices  $u_1,\dots,u_{k}$. Since $H$ has no isolated vertices we conclude that all vertices of $H$ except possibly $u$ and $v$ are one of $u_i$ or their neighbors in $G$. By using this fact and \eqref{eq:degree in H is bounded}, we obtain the following bound:
\begin{align*}
    |H|
&\le 2(k+1) + \sum_{i=k+1}^{s-\lambda}(\deg_H(u_i)+1)  \\
& \le 2(k+1) + \sum_{i=k}^{s-\lambda} (s-\lambda(u_i)+2)\\
&=\OO(s^2).
\end{align*}
A similar count holds for the sum of edge-multiplicities in $H$ as shown next. By~\eqref{eq:number vertices in G from H} and \eqref{eq:degree in H is bounded}, it follows that:
\begin{align*}
   |G| &\le |H| + \lambda + k + \sum_{i=k+1}^{s-\lambda} \lambda(u_i)\deg_H(u_i) \\
&\le |H| + 2s + \sum_{i=k+1}^{s-\lambda} \lambda(u_i)(s-\lambda(u_i)+1) \\
 &\le \OO(s^2) + \sum_{i=k+1}^{s-\lambda} \lambda(u_i)(s-\lambda(u_i)).
\end{align*}
If $\lambda\geq\frac{s}{2}$, then
$$
\sum_{i=k+1}^{s-\lambda} \lambda(u_i)(s-\lambda(u_i))
\le \sum_{i=1}^{s-\lambda}\frac{s^2}{4}
\le \frac{s^3}{8}.
$$
Otherwise, for each $i=k+1,\dots, s-\lambda$, we have
$\lambda(u_i)(s-\lambda(u_i))\le \lambda(s-\lambda)$.
Consequently,
$$
\sum_{i=k+1}^{s-\lambda} \lambda(u_i)(s-\lambda(u_i)) \le  \lambda(s-\lambda)^2 \le \frac{4}{27}s^3,
$$
where the last inequality is attained for $\lambda=\frac{s}{3}$ which maximizes the expression $\lambda(s-\lambda)^2$ for $0\leq\lambda\leq s$. Because $\frac{1}{8}<\frac{4}{27}$, we have
$$
  |G| \le \frac{4}{27}s^3+\OO(s^2),
$$
which establishes the upper bound.
\end{enumerate}
\end{proof}

With a more delicate analysis, we can slightly improve the upper bound in Proposition~\ref{prop:small values}~\ref{tockaD}) and replace the leading constant $\tfrac{4}{27}$ with a smaller constant. As the truth seems to be closer to the lower bound, this does not seem to be worth the effort of making the proof longer.

It would be interesting to prove that, for every fixed $c$, $N(c,s)$ is bounded above by a polynomial in $s$.

\section{Acknowledgements}

We would like to thank Matt DeVos for simplifying our proof of Lemma~\ref{matt} and Fiachra Knox for some valuable discussions.

\bibliographystyle{abbrv}
\bibliography{literature}

\begin{thebibliography}{10}

\bibitem{CsKor11}
I.~Csisz{\'a}r and J.~K{\"o}rner.
\newblock {\em Information theory}.
\newblock Cambridge University Press, second edition, 2011.

\bibitem{CsKoLoMaSi90}
I.~Csisz{\'a}r, J.~K{\"o}rner, L.~Lov{\'a}sz, K.~Marton, and G.~Simonyi.
\newblock Entropy splitting for antiblocking corners and perfect graphs.
\newblock {\em Combinatorica}, 10(1):27--40, 1990.

\bibitem{zanimivaDomneva}
C.~{De Simone} and J.~Körner.
\newblock On the odd cycles of normal graphs.
\newblock {\em Discrete Appl. Math.}, 94(1–3):161--169, 1999.

\bibitem{drugace}
E.~Fachini and J.~Körner.
\newblock Cross-intersecting couples of graphs.
\newblock {\em J. Graph Theory}, 56(2):105--112, 2007.

\bibitem{Ararat}
A.~Harutyunyan, L.~Pastor, and S.~Thomasse.
\newblock Disproving the normal graph conjecture.
\newblock 2015.
\newblock Preprint, arXiv:1508.05487.

\bibitem{Korner73}
J.~K{\"o}rner.
\newblock An extension of the class of perfect graphs.
\newblock {\em Studia Sci. Math. Hungar.}, 8:405--409, 1973.

\bibitem{KorLongo73}
J.~Korner and G.~Longo.
\newblock Two-step encoding for finite sources.
\newblock {\em IEEE Trans. Inf. Theor.}, 19(6):778--782, 2006.

\bibitem{Perfect_couples}
J.~K{\"{o}}rner, G.~Simonyi, and Z.~Tuza.
\newblock Perfect couples of graphs.
\newblock {\em Combinatorica}, 12(2):179--192, 1992.

\bibitem{Lovasz}
L.~Lovász.
\newblock Normal hypergraphs and the perfect graph conjecture.
\newblock {\em Discrete Math.}, 2(3):253--267, 1972.

\bibitem{CitatLovasza}
I.~E. Zverovich and V.~E. Zverovich.
\newblock Basic perfect graphs and their extensions.
\newblock {\em Discrete Math.}, 293(1–3):291--311, 2005.

\end{thebibliography}

\end{document}